\documentclass{amsart}
\usepackage[T1]{fontenc}

\usepackage[foot]{amsaddr}
\usepackage{graphicx}

\usepackage{float}
\usepackage{caption}

\usepackage[x11names]{xcolor}

\usepackage{tikz}
\usetikzlibrary{positioning}
\usetikzlibrary{arrows.meta}
\usepackage{xcolor}

\usepackage{amsmath}
\usepackage{amssymb}
\usepackage{amsthm}
\usepackage{mathtools}
\usepackage{mathrsfs}
\usepackage{stmaryrd}
\usepackage{dsfont}
\usepackage{soul}
\usepackage{cancel}
\usepackage{tcolorbox}
\usepackage{enumitem}
\usepackage[normalem]{ulem}

\definecolor{softblue}{RGB}{30, 60, 180}
\definecolor{softred}{RGB}{180, 30, 30}
\definecolor{softgreen}{RGB}{0, 120, 60}

\usepackage[urlcolor=softred, linkcolor=softblue,citecolor=softgreen, colorlinks=true]{hyperref}

\usepackage{siunitx}

\usepackage{array}
\usepackage{multirow}
\usepackage{tabularx}

\usepackage[T1]{fontenc}
\usepackage[utf8]{inputenc}

\usepackage[numbers,square]{natbib}

\usepackage[margin=1.2in]{geometry}
\usepackage{graphics}

\newtheorem{defin}{Definition}
\newtheorem{prop}{Proposition}
\newtheorem{lemma}{Lemma}
\newtheorem{remark}{Remark}

\newcommand{\R}{\mathbb{R}}

\newcommand{\dt}{\mathrm{d}t}
\newcommand{\dx}{\mathrm{d}x}
\newcommand{\dz}{\mathrm{d}z}
\newcommand{\ind}{\mathds{1}}

\newcommand{\sgn}{\operatorname{sgn}}

\definecolor{myred}{HTML}{eecbcb}
\definecolor{mygreen}{HTML}{cccfc0}

\definecolor{Red}{cmyk}{0,1,1,0.2}

\title[Semi-explicit entropic solutions to a generalized Riemann problem]{Semi-explicit entropic solution to a generalised Riemann problem in some hydrological context}

\author{B. Franke$^{*}$}
\author{M. Lagnaoui$^{*,\dagger}$}
\author{C. Rainer$^{*}$}

\address{$^{*}$ Univ Brest, CNRS UMR 6205,
Laboratoire de Mathématiques de Bretagne
Atlantique}
\email{\href{mailto:majid.lagnaoui@univ-brest.fr}{majid.lagnaoui@univ-brest.fr}}
\thanks{$^{\dagger}$ M. Lagnaoui is the corresponding author. This work was co-funded by Région Bretagne.}

\date{\today}

\begin{document}

\maketitle

\begin{abstract}
We discuss solutions of the one dimensional scalar conservation law with the flux function $y\longmapsto G_{c,\rho}\left(y\right)=((1-\rho)c-y)\mathds{1}_{\{y>c\}}-\rho y\mathds{1}_{\{y\leqslant c\}}$ for two specific initial conditions $u(\cdot,0)=u_0$. This equation arises as the limit of a specific conceptual hydrological model. For initial data strictly below (resp. above) the threshold level $c$, the equation reduces to a constant-speed transport equation with velocity $p$ (resp. $1$). Our goal is to understand precisely what happens when the initial condition crosses the threshold $c$, which corresponds to a generalisation of the Riemann problem, and to provide, in such cases, quasi-closed-form expressions for the corresponding solutions.
\end{abstract}

\vspace{1em}

\section{Introduction}

\noindent
In this paper, we study a scalar conservation law whose flux function is piecewise affine and continuous, featuring a slope change at a threshold value $c$. 
More precisely we consider the following scalar conservation law: 
\[\partial_t u+\partial_x\left(G_{c,\rho}\left(u\right)\right)=0,\]
with intial condition $ u(x,0)=u_0(x) $ and continuous piecewise linear flux function 
\[G_{c,\rho}\left(y\right)=((1-\rho)c-y)\mathds{1}_{\{y>c\}}-\rho y\mathds{1}_{\{y\leqslant c\}}\]
where $\rho\in\left(0,1\right)$ and $c>0$.
  
This equation can, among other interpretations, be understood in a hydrological context, where $u=u(x,t)$ represents a water-related quantity (for example, the water height), $c>0$ is a saturation threshold (or local storage capacity), and $\rho\in]0,1[$ is a transmissivity coefficient (it controls the fraction of flow under the threshold: a small $\rho$ indicates a very retentive medium (high infiltration/storage), while a large $\rho$ corresponds to a highly transmissive one).

The flux function $G_{c,\rho}$ represents two hydrological regimes:

\begin{itemize}
\item Below capacity ($u\leqslant c$): the lateral flow is proportional to the local storage. It means that only a fraction $\rho$ of the stored water is immediately mobilized downstream (drainage), while the rest $1-\rho$ remains in place (infiltration or slow storage).

\item Above capacity ($c<u$): the medium is saturated. The flux is modeling a saturation-excess mechanism. Any additional water contributes entirely to the outflow, apart from a quantity $(1-\rho)c$ that remains.
\end{itemize}

The solution behavior becomes intricate when the initial condition crosses the threshold $c$. In such cases, standard analysis must be adapted to address the slope change in the flux function. Our main goal is to study the solutions of this equation under two types of initial conditions:

\begin{itemize}
    \item In the first scenario, the initial condition undergoes an up-crossing of the threshold $c$.
    \item In the second scenario, the initial condition undergoes a down-crossing of the threshold $c$.
\end{itemize}

\noindent
In both cases, the initial condition crosses the threshold exactly once, which is a generalisation of the classic Riemann problem \cite{leveque2002finite}.

We note that the first scenario, involving an up-crossing, corresponds to a simpler setting where the main difficulty reduces to a pair of interacting linear advection problems at constant speed. As an outcome one then observes a rarefaction phenomenon. In contrast, the down-crossing case gives rise to more intricate dynamics, leading to the occurence of a shock wave, that require careful analysis. In the second case, our method becomes explicit, once one can solve  an implicit function problem obtained from equating suitable integrals of the initial conditions. In both cases, we obtain an entropy weak solution.

\medskip

Scalar conservation laws in one space dimension are the subject of active research \citep{fjordholm2021wellposednesstheorynonlinearscalar} \citep{Friedrich_2022}, \citep{refId0}, \citep{venkatesh2025trackingscalarconservationlaws}, \citep{amir2025hydrodynamicsrelaxationlimitmultilane}
 notably due to their wide range of applications. These include, for instance : fluid flow in pipelines, traffic flow on roads, blood flow in vessels, irrigation channels,  sedimentation processes, etc\dots

Research in this area typically falls into two main categories. The first, which may be described as qualitative, focuses on foundational aspects such as existence and uniqueness of solutions (in an appropriate sense), contraction principles, maximum principles, and regularity of solutions. This also includes the study of shock wave formation, rarefaction phenomena and their large time asymptotics. One can consult the book of Dafermos for an extended overwiev of the field \citep{Dafermos2000}.
The second, more quantitative in nature, concentrates on numerical schemes, in particular, their convergence to the solution and rates of convergence. 

The strong interest in numerical methods stems from the fact that, in most cases, it is not possible to express the solutions of scalar conservation laws explicitly in terms of the problem parameters.
However, we believe that whenever possible, it is worth attempting to derive explicit or closed-form solutions \citep{clamond2025explicitsolutionhyperbolichomogeneous} (as such results may apply to specific configurations of practical interest that are not covered by general theoretical results).
This is precisely the approach we adopt in this paper.

\section{Context and Problem Statement}
\subsection{Problem Statement}
We introduce the Minimal Reservoir Conservation Equation : \[\partial_t u+\partial_x\left(G_{c,\rho}\left(u\right)\right)=0, \;x\in\mathbb{R},\;t\geqslant 0,\] where for fixed $c>0$ and $0<\rho<1$, the flux function is \[G_{c,\rho}\left(y\right)=((1-\rho)c-y)\mathds{1}_{\{y>c\}}-\rho y\mathds{1}_{\{y\leqslant c\}},\;\text{with}\;c>0\;\text{and}\;\rho\in]0,1[.\]

more precisely, we study the Cauchy problem 
\begin{equation}
\begin{cases}
\partial_t u(x,t)+\partial_x\left(G_{c,\rho}\left(u\right)\right)(x,t)=0,\;x\in\mathbb{R},\;t\geqslant 0,\\
u(x,0)=u_0(x),\;x\in\mathbb{R}.
\end{cases}
\label{pbcauchy}    
\end{equation}

for the two  following cases :\begin{enumerate}[label=(\alph*)]
    \item for some fixed $x_0\in\mathbb{R}$, $\begin{cases}
    c<u_0(x),\;\text{for}\;x< x_0,\\
    u_0(x)\leqslant c,\;\text{for}\;x_0<x,
    \end{cases}$
    \item for some fixed $x_0\in\mathbb{R}$, $\begin{cases}
    u_0(x)< c,\;\text{for}\;x< x_0,\\
    c<u_0(x),\;\text{for}\;x_0<x,
    \end{cases}$
\end{enumerate}
\subsection{Some definitions of the general theory}

We start by recalling the notions of weak solutions and weak entropy solutions (\citep{eymard:hal-02100732},
\citep{gallouet_herbin_2024}, \citep{bressan2000hyperbolic}) :
\begin{defin}

\label{definitionnotionsolution}

We consider the one dimensional conservation law :\begin{equation}
    \begin{cases}
    \partial_t u+\partial_x\left(f(u)\right)=0,\\
    u(\cdot,0)=u_0(\cdot),
    \end{cases}
    \label{lcsgeneral}
\end{equation}
where $f$ is locally Lipschitz and $u_0\in L_{\text{loc}}^1\left(\mathbb{R}\right)$.
Let $u\in L^1_{\text{loc}}\left(\mathbb{R}\times\mathbb{R}_+\right)$ such that $f(u)\in L^1_{\text{loc}}\left(\mathbb{R}\times\mathbb{R}_+\right)$.
We say that :
\begin{enumerate}[label=(\roman*)]
    \item The function $u$ is a weak solution of (\ref{lcsgeneral}) if, for all $\phi\in\mathscr{C}^1_c\left(\mathbb{R}\times\mathbb{R}_+\right)$, \[\int_\mathbb{R}\int_0^\infty u(x,t)\partial_t\phi(x,t)+f(u(x,t))\partial_x\phi(x,t)\mathrm{d}t\;\mathrm{d}x+\int_\mathbb{R}u_0(x)\phi\left(x,0\right)\mathrm{d}x=0.\]
    \item The function $u$ is an entropy weak solution of (\ref{lcsgeneral}) if, for every convex function $\eta:\mathbb{R}\to\mathbb{R}$ (called an entropy function), and associated entropy flux function $\Phi\left(s\right)=\int_0^sf^{\prime}\left(\tau\right)\eta^{\prime}\left(\tau\right)\mathrm{d}\tau\;\left(s\in\mathbb{R}\right)$, the following inequality holds for every $\phi\in\mathscr{C}^1_c\left(\mathbb{R}\times\mathbb{R}_+,\mathbb{R}_+\right)$ :
    \begin{equation}
        \int_0^\infty\int_\mathbb{R}\eta\left(u(x,t)\right)\partial_t \phi(x,t)+\Phi\left(u(x,t)\right)\partial_x\phi(x,t)\mathrm{d}x\;\mathrm{d}t+\int_\mathbb{R}\eta\left(u_0(x)\right)\phi\left(x,0\right)\mathrm{d}x\geqslant0 .
        \label{entropyequation}
    \end{equation}
   
\end{enumerate}
\end{defin}

The following result provides an alternative more computational friendly expression for the entropy weak solution.

\begin{prop}[Kr\v{u}zhkov's entropy condition]
\label{Kruzhkoventropies}
The function $u$ is an entropy solution in the sense of the definition (\ref{definitionnotionsolution}). if and only if, for all $k\in\mathbb{R}$ the equation (\ref{entropyequation}) holds with $\eta(s)=|s-k|$ and $\Phi(s)=\sgn(s-k)(f(s)-f(k))$, i.e:
\[\int_\mathbb{R}\int_0^\infty |u(x,t)-k|\partial_t\phi(x,t)+ \sgn\left(u(x,t)-k\right)\left(f(u(x,t))-f(k)\right)\partial_x\phi(x,t)\big)\mathrm{d}t\;\mathrm{d}x+\int_\mathbb{R} |u_0(x)-k|\phi\left(x,0\right)\mathrm{d}x\geqslant 0\]

The function $\eta$ is called Kr\v{u}zhkov's entropy.
\end{prop}

We recall the following results :

\begin{prop} Let $u\in L^1_{\text{loc}}\left(\mathbb{R}\times\mathbb{R}_+\right)$ such that $f(u)\in L^1_{\text{loc}}\left(\mathbb{R}\times\mathbb{R}_+\right)$.
\begin{itemize}
    \item If $u$ is an entropy weak solution of (\ref{lcsgeneral}) then $u$ is a weak solution of (\ref{lcsgeneral}).
    \item If $u$ and $v$ are both entropy weak solutions of (\ref{lcsgeneral}) then $u=v$ a.e.
\end{itemize}
\end{prop}

\subsection{The Constant-Speed Advection Equation}

We end this section of preliminaries by recalling  the well known case of the constant-speed advection equation. It represents the elementary brick of our two-level-equation. For this reason, we detail the results and proofs.

\begin{prop}
For $\rho\geqslant 0$ and $u_0\in L_{\text{loc}}^1\left(\mathbb{R}\right)$ consider the one-dimensional Cauchy transport problem :
\begin{equation}
    \begin{cases}
    \partial_t u-\rho\partial_x u=0,\text{ on }\mathbb{R}\times\mathbb{R}_+^*,\\
    u(x,0)=u_0(x),\text{ }x\in\mathbb{R}.
    \end{cases}
    \label{advectioncauchypb}
\end{equation}
Then the unique weak entropy solution of the problem is given by $(x,t)\mapsto u_0\left(x+\rho t\right)$.
\label{cstspeedadvect}
\end{prop}

\begin{proof}
Let $\eta$ be an entropy function. The associated flux $\Phi$ satisfies, for all $s\in\mathbb{R}$, $\Phi(s)= -\rho(\eta(s)-\eta(0))$.

For all test function $\phi\in\mathscr{C}_c^1\left(\mathbb{R}\times\mathbb{R}_+,\mathbb{R}_+\right)$, it holds that
\begin{align*}
&\int_0^\infty\int_\mathbb{R}\eta(u_0(x+\rho t))\partial_t\phi(x,t)+\Phi(u_0(x+\rho t))\partial_x\phi(x,t)\mathrm{d}x\;\mathrm{d}t\\
&=\int_0^\infty\int_\mathbb{R}\eta(u_0(z))\left(\partial_t\phi(z-\rho t,t)-\rho\partial_x\phi(z-\rho t,t)\right)\mathrm{d}z\;\mathrm{d}t-\rho\eta(0)\int_0^\infty\int_\mathbb{R}\partial_x\phi(z-\rho t,t)\mathrm{d}z\;\mathrm{d}t\\
&=-\int_\mathbb{R}\eta(u_0(z))\phi(z,0)\mathrm{d}z,
\end{align*}
where the last line  follows from the classical chain rule. The result follows.
\end{proof}

\section{First Type of Initial Condition : Down crossing}
\label{crossone}

In this chapter, we solve explicitly the transport equation \eqref{pbcauchy} in the easiest of the two cases, namely when, upstream of the flux, the initial condition is first larger than the level $c$ up to a certain point, and after this point it is smaller than $c$.
More precisely, we introduce two functions $v_0,w_0\in L^\infty(\mathbb{R})$ which satisfy, 
\[ \text{ for all $x\in\R$, } 0\leqslant v_0(x)<c \text{ and } w_0(x)>c,\]
 and suppose that there exists some $x_0\in\mathbb R$ such that the initial condition $u_0$ satisfies

\begin{equation}
	\label{initcondun} 
	u_0(x)=w_0(x)\ind_{\{ x\leqslant x_0\}}+v_0(x)\ind_{\{ x> x_0\}}, \; x\in\R.
\end{equation}

In this setting, the following result holds :

\begin{prop}
The function $u$ defined by : 
\begin{equation}
    u(x,t)=w_0(x+t)\mathds{1}_{(-\infty,x_0-t)}(x)+c\mathds{1}_{[x_0-t,x_0-\rho t]}(x)+v_0(x+\rho t)\mathds{1}_{(x_0-\rho t,\infty)}(x),\;(x,t)\in\mathbb{R}\times\mathbb{R}_+
    \label{firsttype}
\end{equation}
is a entropy weak solution of the problem (\ref{pbcauchy}) with the initial condition (\ref{initcondun}).
\end{prop}

Prior to the proof, we provide a visual representation of the scenario (\ref{firstparamgraphic}).

\begin{figure}[H]
    \centering
\includegraphics[scale=0.8]{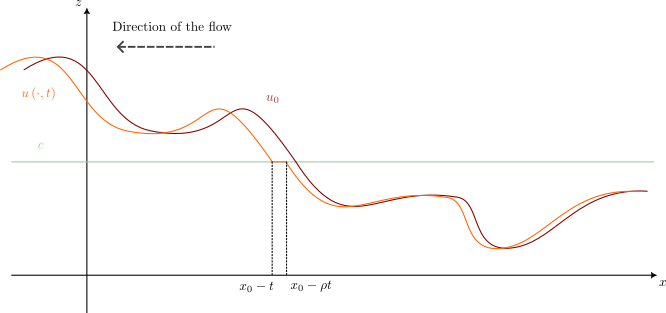}
    \caption{Initial profile and its evolution under the scalar conservation law (Down-Crossing Case). The curve $ u(.,t) $ shows the phenomen of rarefaction on the interval $ ]x_0-t,x_0-\rho t[ $.}
    \label{firstparamgraphic}
\end{figure}

\begin{proof}
Let $\eta$ be an entropy function. The associated flux is given by :\begin{align*}
    \Phi\left(s\right)&=\int_0^s\eta^{\prime}\left(\tau\right)\left(-\rho\mathds{1}_{\tau< c}-\mathds{1}_{c<\tau}\right)\mathrm{d}\tau\\
    &=\begin{cases}
    -\rho\left(\eta(s)-\eta(0)\right)\;\text{if}\;s< c,\\
    -\rho\left(\eta(c)-\eta(0)\right)-\left(\eta(s)-\eta(c)\right)\;\text{if}\;c<s.
    \end{cases}
\end{align*}

We have to prove that  $u$, defined in \eqref{firsttype}, satisfies

\begin{equation}
	\label{entropy1}
\mathscr{E}:=\int_0^\infty\int_\R\big(\eta(u(x,t))\partial_t\phi +\Phi(u)\partial_x\phi\big)\dx\dt +\int_\R \eta(u_0)\phi(\cdot,0)\dx\geqslant0.
\end{equation}
By rewriting $\eta(u)$ and $\Phi(u)$ according to  the expression of $u$ on the different domains, the problem can be split into four distinct constant-speed transport problems.

Applying Proposition \ref{cstspeedadvect} to the transport problem with velocity $(-1)$ (resp. $(-\rho)$, resp. $(-\rho)$, resp. $(-1)$) and initial data $\eta(w_0)\mathds{1}_{(-\infty,x_0)}$ (resp. $\eta(v_0)\mathds{1}_{(x_0,\infty)}$, resp. $\eta(c)\mathds{1}_{(-\infty,x_0)}$, resp. $\eta(c)\mathds{1}_{(-\infty,x_0)}$), we obtain $\mathscr{E}=0$. Thus \eqref{entropy1} is trivially satisfied.
\end{proof}

\section{Second Type of Initial Condition : Up crossing}
\noindent In this section, we analyze how the solution of the equation \eqref{pbcauchy} evolves if, following the direction of the flow, the initial condition is first above the level $c$ and then below. We shall see that it is still possible to compute almost explicitly this solution, but also that this case is much more complex than the previous one.

\subsection{Parametrization of the Initial Condition}
\label{crosstwo}
We consider again the auxiliary functions $v_0$ and $w_0$ which satisfy, for all $x\in\R$, $0\leqslant v_0(x)<c$ and $w_0(x)>c$, and suppose here that $v_0$ and $w_0$ are continuous.\\ 
In this second configuration, in opposition to the first case, we suppose that, for some $x_0\in\mathbb R$  the initial condition $u_0$ satisfies
	\begin{equation}
		\label{initconddeux} 
		u_0(x)=v_0(x)\ind_{\{ x\leqslant x_0\}}+w_0(x)\ind_{\{ x> x_0\}}, \; x\in\R.
\end{equation}

\subsection{Notations and auxilliary functions}

 Here, we propose again an explicit representation of the solution, although it relies on functions defined implicitly through suitable integrals of the intial conditions.

See Figure \ref{twoparamgraphic} for a graphical illustration of our method. We first observe that our specific setting leads the faster left part of the initial profile to catch up with the slower right part. In this process the excess area, which forms between the threshold $ c $ and the profile on the right side, has to match the missing area which forms between the threshold $ c $ and the profile on the left side. This leads to some integral-type constraint relating the position $ x_t $ of the consumption front in the right profile to the location $ y(x_t) $ of the saturation front in the left profile via the difference in propagation speed of the two profiles.
  
\begin{figure}[H]
    \centering
 \includegraphics[scale=0.8]{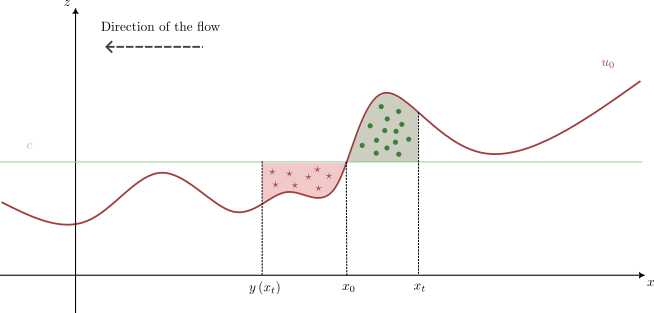}
    \caption{ Illustration of the undderlying idea (Up-Crossing Case).}
  \label{twoparamgraphic}
\end{figure}

We will define two functions $y(\cdot)$ and $x_\cdot$ such that graphically, we have the surface of the red (\textcolor[HTML]{9f3e3e}{$\star$})
 area equals the surface of the green (\textcolor[HTML]{3d7e3b}{$\bullet$})
 area in the graphic (\ref{twoparamgraphic}). 

\subsubsection{Definition of the functions \texorpdfstring{$y(\cdot)$}{y}, \texorpdfstring{$x_\cdot$}{x} and \texorpdfstring{$t(\cdot)$}{t}.}

As a first step, we define some auxiliary functions $I_L$ and $I_R$ by :

\[\begin{array}{rcl}
I_L :[-\infty,x_0] & \longrightarrow & \mathbb{R} \\
y & \longmapsto & \int_y^{x_0}\left(c-v_0\left(z\right)\right)\dz
\end{array}\;\text{and}\;\begin{array}{rcl}
I_R : [x_0,+\infty] & \longrightarrow & \mathbb{R} \\
x & \longmapsto & \int_{x_0}^x\left(w_0\left(z\right)-c\right)\dz.
\end{array}\]

The following result can be easily verified (\citep{leoni2009first}).
 
 \begin{lemma}
 The functions $I_L$ and $I_R$ satisfy the following properties:
 \begin{itemize}
     \item The functions $I_L$ and $I_R$ are both continuously differentiable on their respective domains, and their derivatives are given by : $I_L^{\prime}(y)=v_0\left(y\right)-c\;\text{and}\;I_R^{\prime}(x)=w_0\left(x\right)-c$.
     \item $I_L$ is a strictly decreasing bijection from $(-\infty,x_0]$ onto $[0,I_L(-\infty))$.
     \item $I_R$ is a strictly increasing bijection from $[x_0,\infty)$ onto $[0,I_R(+\infty))$.
 \end{itemize}
 
 \end{lemma}
 
We define the two limiting values
\begin{eqnarray*}
  y_\infty :=  \sup\left\{y<x_0: I_L(y)>I_R(+\infty)\right\} 
    \ \ \ \text{and} \ \ \ 
    x_\infty := \inf\left\{x>x_0: I_R(x)<I_L(-\infty) \right\} ,
\end{eqnarray*}
with the usual the conventions $\sup\emptyset=-\infty$ and $\inf\emptyset=+\infty$.

\begin{remark}
Note that, if $ I_L(-\infty)\leq I_R(+\infty) $, then $ y_\infty=-\infty $, 
	and, if $ I_L(-\infty)\leq I_R(+\infty)) $, then $ x_\infty=+\infty $, so that there are three possibilities : only $x_\infty$ is finite, or only $y_\infty$ is finite, or both are finite.
\end{remark}
 
 The above remark allows us to define the function  $y:[x_0,x_\infty)\rightarrow(y_\infty,x_0]$ as follows:\[y=I_L^{-1}\circ I_R.\]
 
which corresponds, for each $x>x_0$, precisely to the index $y$ on the left of $x_0$ such that, on figure \ref{twoparamgraphic}, the green surface equals the red one.

We define the function $t:[x_0,x_\infty)\rightarrow[0,\infty)$ by the formula \[t(x)=\frac{1}{1-\rho}\left(x-y(x)\right)\]

The following lemma resumes the properties of the functions $y$ and $t$. The proof is elementary, so we omit it.

\begin{lemma} The functions $ y $ and $ t $ defined above have the following properties :
	\begin{enumerate}
		\item 
	The function $y:[x_0,x_\infty)\rightarrow(y_\infty,x_0]$ is strictly decreasing and continuously differentiable, and its derivative is given by :
	
	\begin{equation}
		y^{\prime}(x)=
		-\frac{w_0\left(x\right)-c}{c-v_0\left(y\left(x\right)\right)}.
		\label{y'}
	\end{equation}
	\item The function $ t:[x_0,x_\infty)\rightarrow[0,\infty) $  is  strictly increasing, and continuously differentiable. Its derivative is
	\[ t'(x)= \frac 1{1-\rho}\frac{w_0(x)-v_0(y(x))}{c-v_0(y(x))}.\]
\end{enumerate}
\end{lemma}

Since the functions $y$ and $t$ are bijective, they admit inverse functions. We denote the inverse function of $t$ as
	\[\begin{array}{rrcl}
		x:&[0,\infty)&\rightarrow&[x_0,x_\infty),\\
		&t&\mapsto &x_t.
		\end{array}\]
Further, the inverse function of $ y:[x_0,x_\infty)\rightarrow(y_\infty,x_0] $ will be denoted $ y^{-1}:(y_\infty,x_0]\rightarrow[x_0,x_\infty) $.\\

Now that we have defined  $y(\cdot)$ and $x_\cdot$, we can state a simple yet fundamental Lemma :

\begin{lemma}
The space-time domain is partitioned into the sets $\{x+\rho t< y\left(x_t\right)\}$ and $\{x_t\leqslant x+t\}$. 
\label{noholeslemma}
\end{lemma}

\begin{proof}
This follows from the following sequence of equivalences :
\[y\left(x_t\right)< x+\rho t  \Leftrightarrow x_t+y\left(x_t\right)-x_t< x+\rho t
     \Leftrightarrow x_t-(1-\rho)t< x+\rho t
     \Leftrightarrow x_t< x+t.\]
\end{proof}

\subsubsection{On the regularity and integrability assumptions for $u_0$}
The continuity of $u_0$ is not, strictly speaking, a necessary condition, and we believe it can be relaxed. Any assumption that ensures both the proper definition (and the required properties) of the present section, as well as the validity of the change of variable $z=y(x)$ in the next section, would be suitable.

As an example, assuming that  $u_0$ is locally integrable and that $\mathcal{L}^1\left(\{u_0=c\}\right)=0$ (with $\mathcal{L}^1$ being the Lebesgue measure on $\R$) allows for a well-defined construction of the functions $y(\cdot),\;t(\cdot)$ and $x_\cdot$ together with all the desired properties. However, under this assumption, we were not able to rigorously justify the change of variable in the following section.

This suggests that continuity may be replaced by weaker assumptions. A deeper analysis would be required to determine the minimal set of conditions under which the change of variable remains valid.

\subsection{Weak entropy solution for the Cauchy problem with Down-crossing initial data}

In this section, we use the auxiliary functions defined below to compute an explicit weak entropy solution to the Cauchy problem \eqref{pbcauchy}.

\begin{prop}
The function $u$ defined by : \begin{equation}
u(x,t)=\mathds{1}_{x+\rho t\leqslant y\left(x_t\right)}v_0(x+\rho t)+\mathds{1}_{x_t\leqslant x+t}w_0(x+t),\;\left(x,t\right)\in\mathbb{R}\times[0,\infty)
    \end{equation} 
is an entropy weak solution of the problem (\ref{pbcauchy}) with the initial condition (\ref{initconddeux}).
\label{implicitsolution}
\end{prop}

\begin{proof}
\label{preuveprincipale}

We will use Kr\v{u}zkov's entropy condition as formulated in Proposition \ref{Kruzhkoventropies}.

Thus we need to prove that, $\text{for all}\;\phi\in\mathscr{C}_c^1\left(\mathbb{R}\times\mathbb{R}_+,\mathbb{R}_+\right),\;\text{and}\;k\in\mathbb{R}$ 
\begin{equation}
	\label{entroeq}
	\int_0^\infty\int_\R|u-k|\partial_t\phi+\sgn(u-k)\left(G_{c,\rho}\left(u\right)-G_{c,\rho}\left(k\right)\right)\partial_x\phi\dx\dt +\int_{\R}|u_0(x)-k|\phi(x,0)\dx\geqslant 0.
	\end{equation}

Thanks to Lemma~\ref{noholeslemma}, we can rewrite this equation as follows:
		\begin{equation}
			\label{AB}
			A+B+C\geqslant 0,
		\end{equation}
		with
	\begin{equation}
		\label{A}
	A = \int_\R\int_0^\infty\mathds{1}_{x+\rho t<y\left(x_t\right)}\big(|u(x,t)-k|\partial_t\phi(x,t)+\sgn(u(x,t)-k)\left(G_{c,\rho}\left(u(x,t)\right)-G_{c,\rho}\left(k\right)\right)\partial_x\phi(x,t)\big)\dt\;\dx\end{equation}
	and
	\begin{equation}
		\label{B}
	B=\int_\R\int_0^\infty\mathds{1}_{x_t\leqslant x+t}\big(|u(x,t)-k|\partial_t\phi(x,t)+\sgn(u(x,t)-k)\left(G_{c,\rho}\left(u(x,t)\right)-G_{c,\rho}\left(k\right)\right)\partial_x\phi(x,t)\big)\dt\;\dx,
	\end{equation}
	and
	\begin{equation}
		\label{C}
	 C=\int_\R|u_0(x)-k|\phi(x,0)\dx.
	\end{equation}
	
To prove \eqref{AB}, we examine separately the two terms $A$ and $B$, and finally combine them to the integral $C$.

We start with the simplest situation, where $y_\infty=-\infty$ and $x_\infty=+\infty$, before examining the two other possibilities, where or $x_\infty$ or $y_\infty$ is finite. For each of theses cases, we consider separately the cases $c\leqslant k$ and $k<c$.\\

\noindent Case 1. $y_\infty=-\infty$ and $x_\infty=+\infty$

1.1. $c\leqslant k$

On the domain of the integral $A$,  we have $u(x,t)=v_0(x+\rho t)\leq  c$ and, for $c\leq k$, 
\[ G_{c,\rho}(u(t,x))-G_{c,\rho}(k)=-\rho v_0(x+\rho t)+k+(1-\rho )c=\rho|v_0(x+\rho t)-k|-(1-\rho)(k-c).\]  
Thus, applying the standard change of variable $\psi_\rho(z,t)=\phi(z-\rho t,t)$, we get

\begin{align}
	\label{Ack}
	 A=&
    \int_\mathbb{R}\int_0^\infty \mathds{1}_{z\leqslant y\left(x_t\right)}\big(|v_0\left(z\right)-k|\partial_t\psi_\rho\left(z,t\right)+(1-\rho)(k-c)\left(\partial_x\phi\right)(z-\rho t,t)\big)\dt\;\mathrm{d}z\nonumber\\
    =&\int_{-\infty}^{x_0}\int_0^{t\left(y^{-1}(z)\right)}|v_0\left(z\right)-k|\partial_t\psi_\rho\left(z,t\right)\mathrm{d}t\;\mathrm{d}z+\int_0^\infty\int_{-\infty}^{y\left(x_t\right)}(1-\rho)(k-c)\left(\partial_x\phi\right)(z-\rho t,t)\mathrm{d}z\; \mathrm{d}t \\
    =&\int_{-\infty}^{x_0}|v_0(z)-k|\psi_\rho\left(z,t\left(y^{-1}\left(z\right)\right)\right)\mathrm{d}z-\int_{-\infty}^{x_0}|v_0\left(z\right)-k|\psi_\rho(z,0)\mathrm{d}z \nonumber
    \\
    &+(1-\rho)(k-c)\int_0^\infty\phi\left(y\left(x_t\right)-\rho t,t\right)\mathrm{d}t
\nonumber
\end{align}

Similarly, on the domain of $B$, we have $u(x,t)=w_0(x+t)>c$ and
	$G_{c,\rho}(u(t,x))-G_{c,\rho}(k)=k-w_0(x+t)$. Therefore, setting $\psi(z,t)=\phi(z-t,t)$, we get

\begin{align*}
    B&=\int_\mathbb{R}\int_0^\infty\mathds{1}_{x_t\leqslant z}|w_0(z)-k|\big(\partial_t\phi(z-t,t)-\partial_x\phi(z-t,t)\big)\dt\;\mathrm{d}z\\
    &=\int_{x_0}^{\infty}\int_0^{t(z)}|w_0(z)-k|\partial_t\psi(z,t)\mathrm{d}t\;\mathrm{d}z\\
    &=\int_{x_0}^{\infty}|w_0(z)-k|\psi(z,t(z))\mathrm{d}z-\int_{x_0}^{\infty}|w_0(z)-k|\phi(z,0)\mathrm{d}z.
\end{align*}

Combining these expressions, we get

\begin{align*}
    A+B+C=&\int_{-\infty}^{x_0}|v_0\left(z\right)-k|\psi_\rho\left(z,t\left(y^{-1}(z)\right)\right)\mathrm{d}z+\int_{x_0}^{\infty}|w_0(z)-k|\psi(z,t(z))\mathrm{d}z\\
    &+(1-\rho)(k-c)\int_0^\infty\phi\left(y\left(x_t\right)-\rho t,t\right)\mathrm{d}t.
\end{align*}

Since each of the terms of the right hand side is non negative, relation \eqref{entroeq} is satisfied.

1.2. $k<c$

This case is more difficult to handle. We shall proceed by similar decompositions as previously, but it needs some additional manipulations to obtain our result.
	
Again, we have to prove \eqref{AB} and start by examining separately the two terms $A$ and $B$ defined in \eqref{A} and \eqref{B}.
On the domain of the integral $A$, we have again $u(x,t)=v_0(x+\rho t)\leq c$ but 
\[ G_{c,\rho}(u(x,t))-G_{c,\rho}(k)=-\rho(v_0(x+\rho t)-k),\]

so that, setting again $\psi_\rho(z,t)=\phi(z-\rho t,t)$ we get

\begin{align}
	\label{**}
A=&\int_{-\infty}^{x_0}|v_0(z)-k|\psi_\rho(z,t\left(y^{-1}(z)\right))\mathrm{d}z-\int_{-\infty}^{x_0}|v_0(z)-k|\phi(z,0)\mathrm{d}z\\
=&-\int_{x_0}^{\infty}|v_0\left(y(x)\right)-k|\phi\left(y(x)-\rho t(x),t(x)\right)y^{\prime}(x)\mathrm{d}x-\int_{-\infty}^{x_0}|v_0(z)-k|\phi(z,0)\mathrm{d}z \nonumber\\
=& \int_{x_0}^{\infty}\frac{|v_0\left(y(x)\right)-k|(w_0(x)-c)}{c-v_0(y(x))}\phi\left(x-t(x), t(x),t(x)\right)\mathrm{d}x-\int_{-\infty}^{x_0}|v_0(z)-k|\phi(z,0)\mathrm{d}z,
\nonumber
\end{align}

where the last line follows from Fubini's theorem and the fact that the domain of the integral can be written as $\{ (z,t)\in\R\times\R, z\leq y(x_t)\}=\{ z\leq x_0, 0\leq t\leq t(y^{-1}(z))\}$ (remark that the function $y$ takes its values in $(-\infty,x_0]$).
\\
Again, on the domain of $B$, $u(x,t)=w_0(x+t)>c$ and, with $k\leq c$, we can write here
	\[ G_{c,\rho}(u(x,t)-G_{c,\rho}(k)=-(w_0(x+t)-(1-\rho)c)-\rho c=-|w_0(x+t)-k|+(1-\rho)(c-k), \]
so that $B$ can be decomposed as $B=B_1+B_2$, with 
\begin{align*}
	B_1
	&=\int_{x_0}^\infty\int_0^{t(z)}|w_0(z)-k|\partial_t\psi(z,t)\mathrm{d}t\;\mathrm{d}z\\
	&=\int_{x_0}^\infty|w_0(z)-k|\phi(z-t(z),t(z))\mathrm{d}z-\int_{x_0}^\infty|w_0(z)-k|\phi(z,0)\mathrm{d}z,\\
\end{align*}
and, using an analogue change of variable and inversion of the integrals as for $A$,
 
\begin{align*}
    B_2=& \int_\R\int_0^{\infty}(1-\rho)(c-k)\mathds{1}_{x_t<x+t}\partial_x\phi(x,t)\dt\;\dx\\
    =&-(1-\rho)(c-k)\int_0^\infty\phi\left(x_t-t,t\right)\dt\\
     =&-(1-\rho)(c-k)\int_{x_0}^\infty\phi\left(z-t(z),t(z)\right)t^{\prime}(z)\dz.\\
\end{align*}

Replacing $t'(z)$ in $B_2$ by its expression we obtain 
\[ B=-\int_{x_0}^\infty\frac{(w_0(z)-c)(v_0(y(z))-k)}{c-v_0(y(z))}\phi(z-t(z),t(z))\mathrm{d}z-\int_{x_0}^\infty|w_0(z)-k|\phi(z,0)\mathrm{d}z,\]
so that
\[ A+B+C=\int_{x_0}^\infty\frac{w_0(z)-c}{c-v_0(y(z))}\big(|v_0(y(z))-k|-
(v_0(y(z))-k)\big)\phi(z-t(z),t(z))\dz. \]
It remains to recall that, by definition, $w_0$ takes its values above $c$ and $v_0$ below $c$, to conclude that \eqref{entroeq} is satisfied.\\

\noindent Case 2. $y>-\infty$ and $x_\infty=+\infty$

In this case, to handle the $A$-term, we have to separate the domain $\{ z<y(x_t)\}$ in two subsets: 
\[ \{ z<y(x_t)\}=\{ z< y_\infty\}\cup\{ y_\infty\leq z< x_0, t<t(y^{-1}(z))\}.\]

In the computations for $c\leqslant k$, arriving at relation \eqref{Ack}, we get then
\begin{align*}
	\int_\mathbb{R}\int_0^\infty \mathds{1}_{z\leqslant y\left(x_t\right)}|v_0\left(z\right)-k|\partial_t\psi_\rho\left(z,t\right)\mathrm{d}t\;\mathrm{d}z
	=&\int_{y_\infty}^{x
		_0}|v_0\left(z\right)-k|\psi_\rho\left(z,t\left(y^{-1}(z)\right)\right)\mathrm{d}z\ -\int_{y_\infty}^{x_0}|v_0(z)-k|\psi_\rho(z,0)\mathrm{d}z\\
	& -\int_{-\infty}^{y_\infty}|v_0\left(z\right)-k|\psi_\rho\left(z,0\right)
	\mathrm{d}z.
	\end{align*}

This leads to
\[A=\int_{y_\infty}^{x_0}|v_0\left(z\right)-k|\psi_\rho\left(z,t\left(y^{-1}(z)\right)\right)\mathrm{d}z-\int_{-\infty}^{x_0}|v_0\left(z\right)-k|\phi(z,0)\mathrm{d}z+(1-\rho)(k-c)\int_0^\infty\phi\left(y\left(x_t\right)-\rho t,t\right)\mathrm{d}t.\]
The terms $B$ and $C$ remaining unchanged, we get finally
\begin{align*}
	A+B+C=&\int_{y_\infty}^{x_0}|v_0(z)-k|\psi_\rho\left(z,t\left(y^{-1}\left(z\right)\right)\right)\mathrm{d}z+\int_{x_0}^\infty|w_0(z)-k|\psi(z,t(z))\mathrm{d}z\\
	&\hspace{1.9in} +(1-\rho)(k-c)\int_0^\infty\phi\left(y\left(x_t\right)-\rho t,t\right)\mathrm{d}t\geqslant 0.
\end{align*}
In the same way, for $k<c$, relation \eqref{**} becomes
\begin{align*}
	A =&\int_{y_\infty}^{x_0}|v_0(z)-k|\psi_\rho\left(z,t\left(y^{-1}\left(z\right)\right)\right)\dz-\int_{y_\infty}^{x_0}|v_0(z)-k|\phi(z,0)\dz-\int_{-\infty}^{y_\infty}|v_0(z)-k|\phi(z,0)\dz\\
	=& -\int_{x_0}^{\infty}\frac{|v_0\left(y(x)\right)-k|(w_0(x)-c)}{v_0(y(x))-c}\phi\left(x-t(x), t(x),t(x)\right)\dx-\int_{-\infty}^{x_0}|v_0(z)-k|\phi(z,0)\dz.	
\end{align*}
Again, there are no changes in the $B$ and $C$-terms, so that the conclusion follows as in the first case.

\noindent Case 3. $y=-\infty$ and $x_\infty<+\infty$\\
The arguments are similar: here we use the decomposition $\{ x_t\leq z\}=\{ x_\infty\leq z\}\cup\{ t\leq t(z),x_0\leq z<x_\infty\}$ and get,
for $c\leq k$,
\begin{align*}
	B&=\int_\R\int_0^\infty\mathds{1}_{x_t\leqslant z}|w_0(z)-k|\;\partial_t\psi(z,t)\dt\;\dz\\
		&=\int_{x_0}^{x_\infty}|w_0(z)-k|\psi(z,t(z))\dz-\int_{x_0}^\infty|w_0(z)-k|\phi(z,0)\dz,
\end{align*}
while $A$ and $C$ remain unchanged, so that 
\begin{align*}
	A+B+C=&\int_{-\infty}^{x_0}|v_0(z)-k|\psi_\rho\left(z,t\left(y^{-1}\left(z\right)\right)\right)\mathrm{d}z+\int_{x_0}^{x_\infty}|w_0(z)-k|\psi(z,t(z))\mathrm{d}z\\
	& \hspace{1.9in}  +(1-\rho)(k-c)\int_0^\infty\phi\left(y\left(x_t\right)-\rho t,t\right)\mathrm{d}t
	\geqslant 0.
\end{align*}
For $k<c$,
\[ A= \int_{x_0}^{x_\infty}\frac{|v_0\left(y(x)\right)-k|(w_0(x)-c)}{c-v_0(y(x))}\phi\left(x-t(x), t(x),t(x)\right)\mathrm{d}x-\int_{-\infty}^{x_0}|v_0(z)-k|\phi(z,0)\mathrm{d}z
\]
and
\[ B=-\int_{x_0}^{x_\infty}\frac{(w_0(z)-c)(v_0(y(z))-k)}{c-v_0(y(z))}\phi(z-t(z),t(z))\mathrm{d}z-\int_{x_0}^\infty|w_0(z)-k|\phi(z,0)\mathrm{d}z,\]
and the conclusion still holds.

\end{proof}

\subsection{Examples}
This subsection is devoted to two examples illustrating the method introduced earlier.

The first one is intended to demonstrate that our approach yields an explicit solution in situations where the standard theory of scalar conservation laws does not provide one.

In the second case, we will recall a Riemann problem, through which we show that our method recovers a well-known fundamental result.

\subsubsection{First example}

We consider the initial condition $u_0(x)=\arctan(x-x_0)+c$.

In this case, we get 
\[\begin{cases}
y(x)=2x_0-x,\\t(x)=\frac{2}{1-\rho}\left(x-x_0\right),\\x_t=\frac{1-\rho}{2}t+x_0,
\end{cases}\]

Therefore :\[u(x,t)=\left(\arctan\left(x+\rho t-x_0\right)+c\right)\mathds{1}_{x+\frac{1+\rho}{2}t\leqslant x_0}+\left(\arctan\left(x+t-x_0\right)+c\right)\mathds{1}_{x_0<x+\frac{1+\rho}{2}t}.\]

\subsubsection{Second example}

Let's consider the initial condition $u_0(x)=u_L\mathds{1}_{(-\infty,x_0]} (x)+u_R\mathds{1}_{(x_0,\infty)}(x)$, consider the Riemann problem : \[\begin{cases}
\partial_t u(x,t)+\partial_t G_{c,\rho}\left(u\right)(x,t)=0,\;  x\in\R, t\geq 0,\\u(x,0)=u_0\left(x\right),\;x\in\R,
\end{cases}\] with $u_L<c<u_R$.
Direct computations yield\[\begin{cases}
y(x)=x_0-\left(x-x_0\right)\frac{u_R-c}{c-u_L},\\t(x)=\frac{1}{1-\rho}\left(x-x_0\right)\frac{u_R-u_L}{c-u_L},\\x_t=x_0+\frac{c-u_L}{u_R-u_L}(1-\rho)t.
\end{cases}\]

Therefore : 

\[u(x,t)=\begin{cases}
u_L,\;\text{if}\;x+\left(1-\frac{c-u_L}{u_R-u_L}(1-\rho)\right)t\leqslant x_0,\\u_R,\;\text{if}\;x_0<x+\left(1-\frac{c-u_L}{u_R-u_L}(1-\rho)\right)t.
\end{cases}\]

 Applying the formula  \[1-\frac{c-u_L}{u_R-u_L}(1-\rho)=-\frac{-u_R+(1-\rho)c-(-\rho)u_L}{u_R-u_L}=-\frac{G_{c,\rho}\left(u_R\right)-G_{c,\rho}\left(u_L\right)}{u_R-u_L},\] 
 we recover the classical expression for the shock propagation speed in hyperbolic PDE's (\citep{leveque2002finite}).

\section{Conclusion}

In this work we established an analytical framework for the Minimal Reservoir Conservation Equation (a one-dimensional scalar conservation law with a piecewise-affine flux that changes slope at a threshold level c). Focusing on two prototypical initial configurations that cross the threshold exactly once, we derived explicit entropy weak solutions and verified their admissibility through Kr\v{u}zhkov's entropy formulation. For the down-crossing profile (initially above c then below), the solution can be written as a superposition of two constant-speed advections plus a plateau at level $c$. For the up-crossing profile (initially below c then above), implicit trajectories defined by the functions $y$, $t$ and $x$ yield a closed-form representation. These constructions provide not only existence but also uniqueness thanks to the entropy framework. 

We illustrated the theory with two examples: a smooth arctangent, and a Riemann-type discontinuity. 

A natural continuation of this work is the study of two-dimensional
analogues in which the transport dynamics depend on whether the solution
lies below or above the threshold. One may consider models in which the
sub-threshold part is advected at a reduced proportion of the velocity
field, or models involving distinct velocity fields acting separately on
each regime. While such extensions do not seem to require major
conceptual changes, they raise substantial technical difficulties: the
analogues of $I_L$, $I_R$, $y$, $x$, and $t$  become significantly more
involved, making both explicit representations and entropy-admissibility
checks much more intricate. From a modeling perspective, such a
two-dimensional framework naturally captures runoff dynamics over a
catchment, where the flow is driven by a slope-induced velocity field
and the threshold represents the transition from subsurface transport to
faster overland flow once soil capacity is exceeded. This extension is
currently under investigation.

\bibliographystyle{plainnat}

\bibliography{biblio1.bib}

\end{document}